\documentclass[reqno,11pt]{amsart}

\usepackage{amsmath,amsfonts,amssymb,amsthm,epsfig,color}

\makeatletter
\ifx\SetFigFont\undefined\gdef\SetFigFont#1#2#3#4#5{\reset@font\fontsize{#1}{#2pt}\fontfamily{#3}\fontseries{#4}\fontshape{#5}\selectfont}\fi

\usepackage{ae}

\usepackage{color}

\numberwithin{figure}{section}

\newtheorem{theorem}{Theorem}[section]
\newtheorem{lemma}[theorem]{Lemma}
\newtheorem{proposition}[theorem]{Proposition}

\theoremstyle{definition}

\numberwithin{equation}{section}

\newcommand{\R}{\mathbb{R}}

\newcommand{\eps}{\varepsilon}
\newcommand{\asym}{\beta}
\newcommand{\potent}{\gamma}
\newcommand{\estimate}{A}

\newcommand{\Ha}{{\mathcal{H}}}
\newcommand{\F}{{\mathcal{F}}}

\begin{document}

\title[Strong  Quantitative Isoperimetric inequality]{A strong form of the Quantitative Isoperimetric inequality}

\author{Nicola Fusco}

\author{Vesa Julin}
\address{Dipartimento di Matematica e Applicazioni ``R. Cacciopoli'', Universit\`a degli Studi di Napoli ``Federico II'', Napoli, Italy}
\email{n.fusco@unina.it}
\email{vesa.julin@jyu.fi}

\thanks{}

\subjclass[2000]{}

\begin{abstract} We give a refinement of the quantitative isoperimetric inequality. We prove that the isoperimetric gap controls not only the Fraenkel asymmetry but also the oscillation of the boundary.   
\end{abstract}
\maketitle

\section{Introduction and statement of the results}

In recent years there has been a growing interest in the study of the stability of a large class of geometric and functional inequalities, such as the isoperimetric and the Sobolev inequality. After some early work going back to the beginning of last century the first quantitative version of the isoperimetric inequality in any dimension was proved by Fuglede in \cite{F}. He showed that if $E$ is a {\it nearly spherical set}, i.e., is a Lipschitz set with the barycenter at the origin and the volume of the unit ball $B_1$ such that
\begin{equation}\label{one}
\partial E=\{z(1+u(z)):\,\,z\in\partial B_1\}\,,
\end{equation}
with $\|u\|_{W^{1,\infty}}$ small, then
\begin{equation}\label{two}
\|u\|^2_{W^{1,2}(\partial B_1)}\leq C\left[P(E)-P(B_1)\right]\,.
\end{equation}
Here $P(\cdot)$ denotes the perimeter of a set. From this estimate he was able to deduce that  the perimeter deficit $P(E)-P(B_1)$ controls also the Hausdorff distance between $E$ and $B_1$, whenever $E$ is nearly spherical or convex.
\par
However, Hausdorff distance is too strong when dealing with general sets of finite perimeter and one must replace it (see \cite{H}) by the so called {\it Fraenkel asymmetry } index 
$$
\alpha(E):=\min_{y\in\R^n}\Bigl\{\frac{|E\Delta B_r(y)|}{r^n}:\,\,|B_r|=|E|\Bigr\}\,.
$$
Then,  the  {\it quantitative isoperimetric} inequality states that there exists a constant $C=C(n)$ such that
\begin{equation}\label{three}
\alpha(E)^2\leq CD(E),\,
\end{equation}
where $D(E)$ stands for the {\it isoperimetric deficit}
$$
D(E):=\frac{P(E)-P(B_r)}{r^{n-1}},\qquad \text{with $|B_r|=|E|\,.$}
$$
Note that in this inequality, first proved in \cite{FuMP} with symmetrization techniques, the exponent $2$ on the left hand side is optimal, i.e., it cannot be replaced by any smaller number. Later on Figalli, Maggi and Pratelli in  \cite{FMP} extended \eqref{three} to the anisotropic perimeter via an optimal transportation argument, while a short proof in the case of the standard perimeter has been recently given in \cite{CL} with an argument based on the regularity theory of area almost minimizers.
\par
In this paper we prove a stronger form of the quantitative inequality \eqref{three}. The underlying idea is that the perimeter deficit should control not only the $L^1$ distance between $E$ and some optimal ball, that is the Fraenkel asymmetry, but also  the oscillation of the boundary. 
\par
Let us fix some notation. Given a ball $B_r(y)$, let us denote by $\pi_{y,r}$ the projection of $\R^n\setminus\{y\}$ onto the boundary $\partial B_r(y)$, that is 
$$
\pi_{y,r}(x):=y+r\frac{x-y}{|x-y|}\qquad\text{for all $x\not=y$}
$$
and let us define the asymmetry index as
$$
A(E):=\min_{y\in\R^n}\biggl\{\frac{|E\Delta B_r(y)|}{r^n}+\biggl(\frac{1}{r^{n-1}}\int_{\partial^*E}|\nu_E(x)-\nu_{B_r(y)}(\pi_{y,r}(x))|^2\,d\Ha^{n-1}(x)\biggr)^{1/2}\!:\,\, |B_r|=|E|\biggr\}\,,
$$
where $\partial^*E$ is the reduced boundary of  $E$ and $\nu_E$ is its generalized exterior normal. Then, our main result reads as follows.
\begin{theorem}
\label{mainthm}
Let $n\geq2$. There exists a constant $C(n)$ such that for every set $E\subset\R^n$ of finite perimeter
\begin{equation}\label{main}
A(E)^2\leq CD(E)\,.
\end{equation}
\end{theorem}
A few comments on this inequality are in order. First, let us observe that \eqref{main} is essentially equivalent to  the estimate \eqref{two} for nearly spherical sets. In fact, if $|E|=|B_1|$ and $\partial E$ is as in \eqref{one}, then the normal vector at a point $x(z)=z(1+u(z))$ is given by
$$
\nu_E(x(z))=\frac{z(1+u(z))+\nabla_\tau u(z)}{\sqrt{(1+u)^2+|\nabla_\tau u|^2}}\,,
$$
where $\nabla_\tau u$ stands for the tangential gradient of $u$ on the unit sphere. Therefore, recalling that $\|u\|_{W^{1,\infty}}$ is small, one easily gets
\[
\begin{split}
A(E)^2 & \leq \biggl[|E\Delta B_1|+\biggl(\int_{\partial E}\Bigl|\nu_E(x)-\frac{x}{|x|}\Bigr|^2d\Ha^{n-1}\biggr)^{1/2}\biggr]^2 \\
 & \leq C \biggl[\biggl(\int_{B_1}|u|\,dx\biggr)^2+\int_{\partial B_1}\biggl(1-\frac{1+u(z)}{\sqrt{(1+u)^2+|\nabla_\tau u|^2}}\biggr)d\Ha^{n-1}\biggr] \\
 & \leq C \biggl[\int_{B_1}|u|^2\,dx+\int_{\partial B_1}|\nabla_\tau u|^2\,d\Ha^{n-1}\biggr]\,.
\end{split}
\]
Hence, \eqref{main} follows by combining this inequality with \eqref{two}.
\par
Next observation is that since the second integral in the definition of $A(E)$ behaves like the $L^2$ distance between two gradients, it should control the symmetric difference $|E\Delta B_r(y)|$ as in a Poincar\'e type inequality. This is precisely the statement of the next result.
\begin{proposition}\label{vesaprop} There exists a constant $C(n)$ such that if $E$ is a set of finite perimeter, then
\begin{equation}\label{five}
A(E)+\sqrt{D(E)}\leq C\beta(E)\,,
\end{equation}
where
\begin{equation}\label{defbeta}
\beta(E):=\min_{y\in\R^n}\biggl\{\biggl(\frac{1}{2r^{n-1}}\int_{\partial^*E}|\nu_E(x)-\nu_{B_r(y)}(\pi_{y,r}(x))|^2\,d\Ha^{n-1}(x)\biggr)^{1/2}\!:\,\,|B_r|=|E|\biggr\}\,.
\end{equation}
\end{proposition}
In view of \eqref{five}, the proof of the strong quantitative estimate \eqref{main} reduces to proving that
\begin{equation}\label{six}
\beta(E)^2\leq CD(E)\,,
\end{equation}
for a suitable constant $C=C(n)$. 
\par
In order to prove  this inequality we follow the  strategy introduced  in \cite{CL}  for proving the quantitative inequality stated in \eqref{three}, with some further simplifications due to  \cite{AFM}, where a different isoperimetric problem is considered (see also  \cite{FM} for a similar approach).
\par
The starting point is  the above observation that Fuglede's result implies \eqref{six} for nearly spherical sets. Then, we argue by contradiction assuming that there exists a sequence of equibounded sets $E_k\subset B_{R_0}$, for some $R_0>0$, $|E_k|=|B_1|$, converging in $L^1$ to the unit ball and for which \eqref{six} does not hold. The idea is to replace this sequence by  minimizers $F_k$ of the following penalized problems
$$
\min\bigl\{P(F)+\frac14|\beta(F)^2-\beta(E_k)^2|+\Lambda\bigl||F|-|B_1|\bigr|:\,\,F\subset B_{R_0}\bigr\}\,,
$$
where $\Lambda>n$ is a fixed constant. Then, we show that also $F_k$ converges  in $L^1$ to the unit ball. Moreover, each $F_k$ is an area almost minimizer. Thus, a well known result of B. White (see \cite{W}) yields that the sets $F_k$ actually converge in $C^1$ to the unit ball, and in particular that for $k$ large they are all nearly spherical. This immediately gives a contradiction on observing that if \eqref{six} does not hold for $E_k$, the same is true also  for $F_k$.
\par
We conclude with a final remark. In order to prove the area almost minimality of the sets $F_k$ we have to show preliminarily that they are area quasiminizers. This is a much weaker notion
 than almost minimality (see definition \eqref{quasimin} below), but it is enough to ensure that the sets are uniformly porous (see \cite{DS} and \cite{KKLS}). This mild regularity property turns out to be an essential tool to pass from the $L^1$ convergence to the Hausdorff convergence of the sets.

\section{Preliminaries}
We denote by $B_r(x)$ a ball with radius $r$ centered at $x$ and write  $B_r$ when the center is at the origin. We set $\omega_n:=|B_1|$.  If $E$ is a measurable set in $\R^n$ we denote by $P(E)$ its perimeter and by $\partial^*E$ its reduced boundary. The generalized outer normal will be denoted by $\nu_E$. For the precise definition of these quantities and their main properties we  refer to \cite{AFP}.

A key tool in the proof of Theorem \ref{mainthm} is the result by Fuglede \cite{F}. As  observed in the Introduction, it  implies   \eqref{main}  for Lipschitz sets which are  close to the unit ball in $W^{1, \infty}$. 
\begin{theorem}[Fuglede]
\label{Fuglede}
Suppose that $E\subset\R^n$ has its barycenter at origin, $|E|= \omega_n$,   and that
\[
\partial E= \{ ( 1 + u(z) ) \, z: \,\, z \in \partial B_1\}
\]
for $u \in W^{1, \infty}(\partial B_1)$. There exist $c>0$ and $\eps_0 >0$ such that if $ || u||_{W^{1, \infty}(\partial B_1)} \leq \eps_0$, then
\[
D(E) \geq c \, || u||_{W^{1, 2}(\partial B_1)}^2\,.
\]  
 Moreover,
\begin{equation}\label{fuglede1}
\beta(E)^2\leq A(E)^2\leq C_0D(E)\,,
\end{equation}
for some positive constant $C_0$ depending only on $n$.
\end{theorem}

Another key ingredient in our proof is the regularity of   area almost minimizers. To this aim, we recall that a set $F$ is an area {\it $( \Lambda, r_0)$-almost minimizer}  if for every $G$, such that $G \Delta F \Subset B_r(x)$ with $r \leq r_0$, it holds
\[
P(F) \leq P(G) + \Lambda r^n.
\]
Next result is contained in \cite{W}.
\begin{theorem}[B. White] 
\label{areamin}
Suppose that $F_k$ is a sequence of  area $(\Lambda, r_0)$-almost minimizers  such that 
\[
\sup_k \, P(F_k) < \infty  \quad \text{and} \quad \chi_{F_k} \to \chi_{ B_1} \quad \text{in} \,\,L^1. 
\] 
Then, for $k$ large, each $F_k$ is of class $C^{1, \frac{1}{2}}$ and 
\[
\partial F_k = \{ (1+u_k(z))\,z \mid z \in \partial B_1 \}\,,
\]
with $u_k \to 0$ in $ C^{1, \alpha}(\partial B_1)$ for every $ \alpha \in (0, \frac{1}{2})$. 
\end{theorem}

We will also use the theory of the so called area $(K, r_0)$-quasiminimizers.  We say that a set $F$ is an area  {\it $(K, r_0)$-quasiminimizer}  if for every $G$, such that $ G \Delta F \Subset   B_{r}(x)$ with $r \leq r_0$, the following inequality holds
\begin{equation}\label{quasimin}
P(F; B_{r}(x) ) \leq K \, P(G; B_{r}(x)) \,.
\end{equation}
Here $P(G; B_{r}(x)) $ stands for the perimeter of $G$  in $B_{r}(x)$.

The regularity of $(K, r_0)$-quasiminimizers is very weak. Nevertheless we have the following result by David and Semmes \cite{DS}, see also  Kinnunen, Korte, Lorent and Shanmugalingam \cite{KKLS}, where the result below is proven in a general metric space. 
\begin{theorem}[David \& Semmes] 
\label{david}
Suppose that $F$ is an area $(K, r_0)$-quasiminimizer. Then, up to modifying $F$ in a set of measure zero, the topological boundary of $F$ coincides with the reduced boundary, i.e., 
$\partial F = \partial^* F$. 
\par
Moreover $F$ and $\R^n \setminus F$ are locally porous, i.e.,  there exist $R> 0 $ and $C>1$ such that for any $0 < r< R$ and every $x \in \partial F$ there are points $y,z \in B_r(x)$ for which
\[
B_{r/C}(y) \subset F \qquad \text{and} \qquad B_{r /C}(z) \subset \R^n \setminus F.
\] 
\label{david.semmes}

\end{theorem}

\section{Proof of the theorem}
In this section we give the proof of Theorem~\ref{mainthm}. Since the  quantities $A(E)$ and $D(E)$  in \eqref{main} are scale invariant, we shall assume from now on and without loss of generality that $|E|=\omega_n$. Moreover, in view of Proposition~\ref{vesaprop}, whose proof will be given at the end of this section, we will only need  to prove the estimate \eqref{six}. 
\par
Thus, we begin by giving a closer look to the oscillation term $\beta(E)$ defined in \eqref{defbeta}. Observe that by the divergence theorem we immediately have 
\[
\begin{split}
\frac12\int_{\partial^*E}|\nu_E(x)-\nu_{B_r(y)}(\pi_{y,r}(x))|^2\,d\Ha^{n-1}(x) & =\int_{\partial^*E}\Bigl(1-\nu_E(x)\cdot\frac{x-y}{|x-y|}\Bigr)\,d\Ha^{n-1}(x) \\
& = P(E)-\int_E\frac{n-1}{|x-y|}\,dx\,.
\end{split}
\]
Therefore, we may write
\begin{equation}
\label{remember}
\asym (E)^2 = P(E) -(n-1)\potent(E)\,,
\end{equation}
where we have set
\begin{equation}
\label{potent}
\potent (E) := \max_{y\in\R^n}   \int_E \frac{1}{|x-y|} \, dx\,.
\end{equation}
We say that a set $E$ is \emph{centered at $y$} if 
\[
\asym(E)^2 =\int_{\partial^* E} \left( 1- \nu_{E} \cdot \frac{x-y}{|x-y|} \right) \, d \Ha^{n-1}(x).
\]
Notice that in general a center of a set is not unique. 

The following simple lemma shows that the centers  of sets, which are close to the  unit ball in $L^1$, are close to the origin.
\begin{lemma}
\label{centerpoint}
For every $\eps>0$ there exists $\delta>0$ such that if  $F \subset B_{R_0}$ and  $|F \Delta B_1| < \delta$, then $|y_F| < \eps$ for every center  $y_F$ of $F$.
\end{lemma}

\begin{proof}
We argue by contradiction and assume that there exist $F_k \subset B_{R_0}$ such that  $F_k \to B_1$ in $L^1$ and $y_{F_k} \to y_0$ with $|y_0| \geq \eps$, for some $\eps >0$. Then we would have
\[
\int_{F_k} \, \frac{1}{|x|} \, dx \leq \int_{F_k} \, \frac{1}{|x - y_{F_k}|} \, dx.
\] 
Letting $k \to \infty$,  by the dominated convergence theorem the left hand side converges to $\int_{B_1} \, \frac{1}{|x|} \, dx$, while the right hand side converges to  $\int_{B_1} \, \frac{1}{|x - y_0|} \, dx$. Thus we have
\[
\int_{B_1} \, \frac{1}{|x|} \, dx \leq \int_{B_1} \, \frac{1}{|x - y_0|} \, dx.
\]
By the divergence theorem we conclude that
\[
\int_{\partial B_1} 1 \, dx \leq \int_{\partial B_1} \, x \cdot \frac{x- y_0}{|x - y_0|}  \, dx
\]
and this inequality may only hold if $y_0=0$, thus leading to a contradiction.
\end{proof}

The next lemma states that in order to prove \eqref{six} we may always assume that our set $E$ is contained a sufficiently large  ball $B_{R_0}$. The proof follows closely the one given in \cite[Lemma~5.1]{FMP} and we only indicate the few changes needed in our case.
\begin{lemma}
\label{large.ball}
There exist $R_0 =R_0(n)>1$ and $C =C(n)$ such that for every set $E$, with $|E| = \omega_n$, we may find $E' \subset B_{R_0}$ such that $|E'| = \omega_n$ and
\begin{equation}\label{largeball1}
 \asym(E)^2 \leq \asym(E')^2 + CD(E), \qquad D(E') \leq C D(E)\,.
\end{equation}
\end{lemma}

\begin{proof}
Let us assume that $D(E) \leq \frac{1}{4}(2^{1 / n} -1) $. Otherwise, observing that $\beta(E)^2\leq P(E)= P(B_1)+D(E)$, \eqref{largeball1} (and in turn \eqref{six}) follows at once taking $E'=B_1$ and a sufficiently large constant $C(n)$. 
\par
Moreover, up to rotation, we may also assume
without loss of generality  that 
\[
\Ha^{n-1 }(\{ x \mid \nu_E(x) = \pm e_i  \}) = 0
\]
for any $i = 1, \dots, n$. 

Arguing exactly as in \cite{FMP}, we may find $\tau_1, \tau_2$ such that $0 < \tau_2 -\tau_1 < \rho_0$, for some $\rho_0$ depending only on $n$, such that the set $\tilde{E} = E \cap \{x \mid \tau_1 < x_1 < \tau_2  \}$ satisfies
\begin{equation}
\label{Etilde}
|\tilde{E}| \geq |B_1| \left( 1- 2 \, \frac{D(E)}{2^{1 / n} -1} \right) \qquad \text{and} \qquad P(\tilde{E}) \leq P(E).
\end{equation}
The latter inequality follows simply from the fact that we cut $E$ by a hyperplane. 

The first inequality in \eqref{Etilde} and the isoperimetric inequality yield
\[
P(\tilde{E}) \geq n\omega_n^{1/n} |\tilde{E}|^{\frac{n-1}{n}} \geq n\omega_n^{1/n} |B_1|^{\frac{n-1}{n}} \left( 1- 2 \frac{D(E)}{2^{1 / n} -1} \right)^{\frac{n-1}{n}} \geq  P(B_1) \left( 1- C \, D(E) \right).
\]
From this inequality, using \eqref{remember} and \eqref{potent} and denoting by $y_{\tilde E}$ the center of $\tilde E$, we get
\begin{equation}
\label{asym-estim}
\begin{split}
 \asym(E)^2 -\asym(\tilde{E})^2 &\leq (  P(E) - P(\tilde{E}) ) + \int_{\tilde{E}} \frac{n-1}{|x-y_{\tilde E}|} \,  dx - \int_{E} \frac{n-1}{|x-y_{\tilde E}|} \,  dx \\
&\leq  D(E) +  P(B_1)-P({\tilde E})
\leq C_1 D(E).
\end{split}
\end{equation}
Set  now
\[
\lambda = \left( \frac{|B_1|}{|\tilde{E}|} \right)^{1/n} \qquad \text{and} \qquad E' = \lambda \tilde{E}.
\]
From the first inequality in \eqref{Etilde} we get that $1 \leq  \lambda \leq 1 + C_2 D(E)$, while the second inequality   yields
\[
P(E') = \lambda^{n-1} P(\tilde{E}) \leq (1 + C_3 D(E)) P(E)
\]
and the second inequality in \eqref{largeball1} follows. On the other hand from \eqref{asym-estim} we get
\[
\asym(E')^2 = \lambda^{n-1} \asym(\tilde{E})^2 \geq \asym(\tilde{E})^2 \geq \asym(E)^2 - C_1 D(E)\,,
\]
that is  the first inequality in \eqref{largeball1}.
The proof is completed by repeating the same argument for all coordinate axes.
\end{proof}

We will also need the following corollary of the isoperimetric inequality. 
\begin{lemma}
\label{penal.ball}
Suppose that $R_0>1$ and $\Lambda >n$. Then, up to a translation, the unit ball $B_1$ is the unique minimizer of 
\[
P(F) + \Lambda \big| |F| - \omega_n \big| 
\] 
among all sets contained in $B_{R_0}$.
\end{lemma}

\begin{proof}
Suppose that $E$ is a minimizer of the functional  above . Then we have
\[
P(E) \leq P(E)+ \Lambda \big| |E| - \omega_n \big| \leq P(B_1) =  n \omega_n.
\]
Thus the isoperimetric inequality implies that $|E| \leq \omega_n$. Therefore, by the minimality of $E$ and the isoperimetric inequality again, we have
\[
\begin{split}
0 &\geq P(E) + \Lambda \big| |E| - \omega_n \big|- P(B_1) \\
&\geq n \omega_n^{1/n} \, |E|^{\frac{n-1}{n}} + \Lambda ( \omega_n - |E|) - n \omega_n 
\geq  (\Lambda - n) (\omega_n - |E|)  \,.
\end{split}
\]  
Hence, $E$ is a ball of radius one. 
\end{proof}

The following  lower semicontinuity lemma  will be used in the  proof of Theorem~\ref{main}. It deals with the functional 
\begin{equation} 
\label{functional}
\F(E) =  P(E) + \Lambda \big| |E| - \omega_n  \big| + \frac{1}{4}| \asym (E)^2 - \eps|\,.
\end{equation}

\begin{lemma}
\label{lowersemi}
The functional \eqref{functional} is lower semicontinuous with respect to the $L^1$-convergence in $B_{R_0}$.
\end{lemma}

\begin{proof}Let us first prove that the functional $\potent$ defined in \eqref{potent}  is continuous with respect to $L^1$ convergence in $B_{R_0}$, that is 
\begin{equation}
\label{potent.cont}
\lim_{k\to \infty} \potent (E_k) = \potent(E)\,,
\end{equation}
whenever  $E_k,E\subset B_{R_0}$ and $E_k \to E$ in $L^1$. To this aim, suppose that the sets $E_k$ and $E$ are centered at $y_k$ and at $y_0$, respectively. From  the definition of $\potent$ we obtain
\[
\int_{E_k} \frac{1}{|x-y_0|} \, dx \leq \potent(E_k) 
\]
and therefore $\lim_{k\to \infty} \potent (E_k) \geq \potent(E)\,.$ On the other hand, choose $r_k$ such that $|B_{r_k}| = |E_k \setminus E|$ and use the divergence theorem to obtain 
\[
\begin{split}
\potent(E_k) &\leq \int_{E} \frac{1}{|x-y_k|} \, dx + \int_{ E_k \setminus E} \frac{1}{|x-y_k|} \, dx \leq \potent(E) +  \int_{B_{r_k}(y_k)} \frac{1}{|x-y_k|} \, dx\\&= \potent(E) +  \frac{1}{n-1} P(B_{r_k}) .
\end{split}
\]
Therefore $\lim_{k\to \infty} \potent (E_k) \leq \potent(E)$  and \eqref{potent.cont} follows.
 
To show the lower semicontinuity of $\F$, let us consider   $E_k,E\subset B_{R_0}$, with $E_k \to E$ in $L^1$. Without loss of generality, we may assume that 
\[
\lim \inf_{k\to \infty} \F(E_k) = \lim_{k\to \infty} \F(E_k)<\infty\,.  
\]
Passing possibly to a subsequence we may also assume that   $\lim_{k\to \infty} P(E_k) = \alpha$. By the lower semicontinuity of the perimeter we have
\[
\alpha \geq P(E)\,.
\]
Then, recalling \eqref{remember}, we get
\[
\begin{split}
\lim_{k\to \infty} \F(E_k) &= \alpha + \Lambda  \big| |E| - \omega_n  \big| + \frac{1}{4}| \alpha  - (n-1) \gamma(E) - \eps| \\
&\geq \F(E) + (\alpha - P(E)) - \frac{1}{4}|\alpha - P(E)|  \geq \F(E)\,,
\end{split}
\]
thus concluding  the proof.
\end{proof}

We are now ready to prove the main result.

\begin{proof}[\textbf{Proof of  Theorem \ref{mainthm}}]
Let $c_0>0 $ be a constant which will be chosen at the end of the proof. Thanks to Lemma~ \ref{large.ball} and Proposition~\ref{vesaprop}  it is sufficient to prove that there exists $\delta_0>0$ such that, if $D(E) \leq \delta_0$ then 
$$
D(E)\geq c_0 \asym(E)^2
$$
for all $ E \subset B_{R_0},   |E| = \omega_n.$

We argue by contradiction and assume that there exists a sequence of sets $E_k \subset B_{R_0}$ such that $|E_k| = \omega_n$, $P (E_k) \to P(B_1)$ and
\begin{equation}
\label{contradict}
P(E_k) < P(B_1) + c_0 \, \asym (E_k)^2.
\end{equation}
By compactness we have that, up to a subsequence, $ E_k \to E_{\infty}$ in $L^1$ and by the lower semicontinuity of the perimeter we immediately conclude that $ E_{\infty}$ is a ball of radius one. Set $\eps_k := \asym (E_k)$. 
In the proof of the Lemma \ref{lowersemi} it was shown that the functional $\gamma$ defined in \eqref{potent} is continuous with respect to $L^1$ convergence. Therefore, since  $E_k$ is converging to a ball of radius one in $L^1$ and $P (E_k) \to P(B_1)$, we have that  
\[
\eps_k = P(E_k) -(n-1)\, \potent(E_k) \to 0 
\]
As in \cite{AFM} we replace each set $E_k$ by a minimizer $F_k$ of the following problem 
\begin{equation}
\label{contr-funct}
\min \, \{ P(F) + \Lambda \big| |F| - \omega_n  \big| + \frac{1}{4}| \asym (F)^2 - \eps_k^2|, \quad  F \subset B_{R_0}\}
\end{equation}
for some fixed $\Lambda > n$. By  Lemma~ \ref{lowersemi} we know that the functional above is lower semicontinuous with respect to $L^1$-convergence of sets and therefore a minimizer exists.

\noindent\textbf{Step 1:} \, Up to a subsequence, we may assume that $F_k \to F_{\infty}$ in $L^1$. Since $F_k$ minimizes \eqref{contr-funct} we have from \eqref{contradict} and Lemma \ref{penal.ball} that
\[
\begin{split}
P(F_k) + \Lambda \big| |F_k| - \omega_n  \big|+  \frac{1}{4}| \asym (F_k)^2 - \eps_k^2| &\leq P(E_k) < P(B_1) +  c_0 \, \eps_k^2 \\
&\leq P(F_k) + \Lambda \big| |F_k| - \omega_n  \big| +  c_0 \, \eps_k^2\,.
\end{split}
\]
Hence $| \asym (F_k)^2 - \eps_k^2| \leq 4 c_0 \, \eps_k^2$, which implies $\asym (F_k) \to 0$ and
\begin{equation}
\label{good.obs}
\eps_k^2 \leq \frac{1}{1- 4c_0} \asym (F_k)^2\,.
\end{equation}
Therefore $F_{\infty}$ is a  minimizer of the problem  
\[
\min \{ P(F)  + \Lambda \big| |F| - \omega_n  \big|: \,\,  F \subset B_{R_0}\}\,.
\]
Thus by the Lemma \ref{large.ball} we conclude that  $F_{\infty}$ is a ball $B_1(x_0)$ for some $x_0$.

\noindent\textbf{Step 2:} \,We claim that for any $\eps>0$,  $B_{1- \eps}(x_0) \subset F_k \subset B_{1+\eps}(x_0) $ for $k$ large enough.

To this aim we show that the sets $F_k$ are area $(K, r_0)$-quasiminimizers  and use Theorem~\ref{david.semmes}. Let $G\subset\R^n$ be such that  $G \Delta F_k \Subset  B_r(x)$, $r \leq r_0$. 

\textit{Case 1:}  Suppose  that $B_r(x) \subset B_{R_0}$. By the minimality of $F_k$ we obtain
\begin{equation}
\label{ach}
P (F_k) \leq P(G ) + \frac{1}{4}| \asym(F_k)^2 - \asym(G)^2| + \Lambda \big| |F_k|  - |G| \big| 
\end{equation}
Assume that $\beta(F_k)\geq\beta(G)$ (otherwise the argument is similar) and denote by $y_G$ the center of $G$. Then we get
\[
\begin{split}
&| \asym(F_k)^2 - \asym(G)^2| \leq \int_{\partial^*F_k}\!\left( 1- \nu_{F_k} \cdot \frac{z-y_G}{|z-y_G|} \right)  d \Ha^{n-1}(z)-\int_{\partial^* G}\! \left( 1- \nu_{G} \cdot \frac{z-y_G}{|z-y_G|} \right) d \Ha^{n-1}(z) \\
&\qquad\quad=  \int_{\partial^* F_k\cap B_r(x)}\! \left( 1- \nu_{F_k} \cdot \frac{z-y_G}{|z-y_G|} \right)  d \Ha^{n-1}(z)-\int_{\partial^* G\cap B_r(x)}\! \left( 1- \nu_{G} \cdot \frac{z-y_G}{|z-y_G|} \right)  d \Ha^{n-1}(z) \\
&\qquad\quad\leq 2\bigl[P(F_k ;  B_r(x)) + P(G ;  B_r(x)) ]\,,
\end{split}
\]
where $P(E;  B_r(x))$ stands for the perimeter of $E$ in  $B_r(x)$.  Therefore, from \eqref{ach} we get 
$$
P(F_k ;  B_r(x)) \leq 3P(G ;  B_r(x))+2\Lambda \big| |F_k|  - |G| \big| \,.
$$
From the above inequality the  $(K, r_0)$-quasiminimality immediately follows by observing that  
\[
 |F_k \Delta G| \leq \omega_n^{1/n} r^{1/n} |F_k \Delta G|^{\frac{n-1}{n}} \leq C(n) r^{1/n} [  P(F_k ;  B_r(x)) + P(G ;  B_r(x)) ]
\]
and choosing $r_0 $ sufficiently small.

\textit{Case 2:} If $ |B_r(x) \setminus B_{R_0} | >0$, we may write
\[
\begin{split}
&P(F_k ; B_r(x)) - P(G ; B_r(x)) \\
&=  P(F_k ; B_r(x)) -  P(G \cap B_{R_0} ; B_r(x)) + P(G \cap B_{R_0} ; B_r(x)) - P(G ; B_r(x)) \\
&=  P(F_k; B_r(x)) -  P(G \cap B_{R_0} ; B_r(x))  + P (B_{R_0}) -  P(G \cup B_{R_0}) \\
&\leq P(F_k; B_r(x)) -  P(G \cap B_{R_0} ; B_r(x)).
\end{split}
\]
From Case 1 we have that this term is less than $(K-1) P(G  \cap B_{R_0}; B_r(x))$ which in turn is smaller than  $(K-1) P(G; B_r(x))$.  Hence, all $F_k$ are  $(K, r_0)$-quasiminimizer with  uniform constants $K$ and $r_0$. 

The claim then follows from the theory of $(K, r_0)$-quasiminimizers and the fact that $F_k \to B_1(x_0)$ in $L^1$. Indeed, arguing by contradiction, assume that there exists $0<\eps_0<2r_0$ such that for infinitely many  $k$  one can find $x_k \in \partial F_k $ for which
\[
x_k \notin B_{1+ \eps_0}(x_0) \setminus B_{1- \eps_0}(x_0) . 
\] 
Let us assume that $x_k \in B_{1- \eps_0}(x_0)$ for infinitely many $k$ (otherwise, the  argument  is similar). From Theorem \ref{david} it follows that there exist $y_k \in B_{\frac{\eps_0}{2}}(x_k)$ such that $B_{\frac{\eps_0}{2C}}(y_k) \subset  B_1(x_0) \setminus F_k$. This implies 
\[
|B_1(x_0) \setminus F_k| \geq |B_{\frac{\eps_0}{2C}}| >0, 
\]
which contradicts the fact that $F_k \to B_1(x_0)$ in $L^1$, thus proving the claim.

\noindent\textbf{Step 3:} 
Let us now translate  $F_k$, for $k$ large, so that the resulting sets, still denoted by $F_k$, are contained in $B_{R_0}$, have their barycenters at the origin and converge to $B_1$.
We are going to use Theorem~\ref{areamin}  to show that $F_k$ are $C^{1,1/2}$ and converge to $B_1$ in $C^{1, \alpha}$ for all $\alpha<1/2$. To this aim,  fix a small $\eps>0$. From Step 2 we have that for $k$ large 
\begin{equation}\label{incl}
B_{1-\eps}\subset F_k\subset B_{1+\eps}\,.
\end{equation}
We want to show that when $k$ is large $F_k$ is a $(\Lambda',r_0)$-almost minimizer for some constants $\Lambda',r_0$ to be chosen independently of $k$.

To this aim,
fix  a set  $G\subset\R^n$ such that $G \Delta F_k \Subset B_r(y) $, with $r<r_0$. 

 If $B_r(y)  \subset B_{1-  \eps}$, from \eqref{incl} it follows that $G \Delta F_k  \Subset F_k$ for $k$ large enough. This immediately yields $P(F_k) \leq P(G)$.
 
If  $B_r(y) \not \subset B_{1-  \eps}$, choosing $r_0$ and $\eps$ sufficiently small we have that
\begin{equation}\label{empty}
B_r(y)\cap B_{1/2}=\emptyset\,.
\end{equation}  
Denote by $y_{F_k}$ and  $y_G$ the centers of $F_k$ and $G$, respectively.   If $\eps$ is sufficiently small, from \eqref{incl} and  Lemma \ref{centerpoint} we have that for $k$ large
\begin{equation}
\label{not-far}
|y_{F_k} | \leq \frac14 \quad \text{and} \quad |y_G | \leq \frac14\,.
\end{equation}
By the minimality of $F_k$ we have
\[
P(F_k) \leq P(G) + \frac{1}{4}| P(F_k) - P(G)|   + \Lambda  \big| |F_k| - |G|  \big| + \frac{n-1}{4} |\potent(F_k) - \potent(G) |,
\]
which immediately implies
\begin{equation}
\label{compare}
P(F_k) \leq P(G)    +  2 \Lambda  | F_k \Delta G|+  (n-1)|\potent(F_k) - \potent(G) |.
\end{equation}
We may estimate the last term simply by
\[
\potent(F_k) - \potent(G) \leq \int_{F_k} \frac{1}{|x-y_{F_k}|} \, dx -\int_{G} \frac{1}{|x-y_{F_k}|} \, dx \leq \int_{F_k \Delta G} \frac{1}{|x-y_{F_k}|} \, dx
\]
and 
\[
\potent(G) - \potent(F_k) \leq \int_{G} \frac{1}{|x-y_G|} \, dx -\int_{F} \frac{1}{|x-y_G|} \, dx \leq \int_{F_k \Delta G} \frac{1}{|x-y_G|} \, dx.
\]
Therefore, recalling \eqref{empty} and  \eqref{not-far},  we have
$$
|\potent(F_k) - \potent(G) | \leq  4|F_k \Delta G|\,.
$$
From this estimate and inequality \eqref{compare} we may then conclude that
\[
P(F_k) \leq  P(G) + ( 2\Lambda + 4(n-1)) \, |F_k \Delta G| \leq  P(G) + \Lambda' \, r^n.
\]
Hence, the sets $F_k$ are  $(\Lambda', r_0)$- almost minimizers with uniform constants $\Lambda'$ and $ r_0$.

Thus, Theorem \ref{areamin}  yields that the $F_k$ are $C^{1, 1/2}$ and that, for $k$ large, \begin{equation}
\label{C^1-conver}
\partial F_k = \{ (1 + u_k(z))z:\,\, z \in \partial B_1\}
\end{equation}
for some $u_k\in C^{1,1/2}(\partial B_1)$ such that $u_k \to 0$ in $C^{1}(\partial B_1)$. 

\smallskip

\noindent\textbf{Step 4:}
By the minimality of $F_k$, \eqref{contradict} and  \eqref{good.obs} we have
\begin{equation}
\label{almost.there}
P(F_k) + \Lambda \big|  |F_k| - \omega_n \big| \leq P(E_k) <   P(B_1) + c_0 \eps_k ^2\leq P(B_1) + \frac{c_0}{1 -4c_0} \asym(F_k)^2\,.
\end{equation}
We are almost in a position to use Theorem \ref{Fuglede} to obtain a contradiction. We only need to rescale $F_k$ so that  the volume constrain is satisfied. Thus, set $F_k' := \lambda_k F_k$, where  $\lambda_k$  is such that $\lambda_k ^n|F_k| = \omega_n $. Then $\lambda_k \to 1 $ and also the sets $F_k' $ converge to $B_1$ in $C^1$  and have their barycenters at the origin. Therefore, since $\Lambda>n$, $P(F_k)\to n\omega_n$ and $|F_k|\to\omega_n$, we have that for $k$ sufficiently large
\begin{equation}
\label{scaling}
| P(F_k') - P(F_k) | = |\lambda_k^{n-1} - 1| \, P(F_k) \leq  \Lambda \,  |\lambda_k^n - 1 | \, |F_k|=  \Lambda \, \big| |F_k'| - |F_k|  \big|.
\end{equation}
Then  \eqref{almost.there} and \eqref{scaling} yield
\[
\begin{split}
P(F_k') &\leq P(F_k) + \Lambda  \big|  |F_k| - \omega_n \big| <  P(B_1) + \frac{c_0}{1 -4c_0} \asym(F_k)^2 \\
&=  P(B_1) + \frac{c_0 \, \lambda_k^{1-n}}{1 -4c_0} \asym(F_k')^2.
\end{split}
\]
which  contradicts \eqref{fuglede1} if  $2c_0/(1-4c_0)<1/C_0$ and $k$ is large.
\end{proof}

We conclude by proving that the oscillation index $\asym(E)$ defined in \eqref{defbeta} controls the total asymmetry $\estimate(E)$. 

\begin{proof}[\textbf{Proof of  Proposition~\ref{vesaprop}}]
Let $E$ be a set of finite perimeter such that $|E| = \omega_n$ and assume that
 $E$ is centered at the origin, i.e.,
\[
\asym(E)^2 =  \int_{\partial^* E} \left( 1- \nu_{E} \cdot \frac{x}{|x|} \right) \, d \Ha^{n-1}.
\]
By the divergence theorem we may write
\[
\begin{split}
 \int_{\partial^* E}  \nu_{E} \cdot \frac{x}{|x|} \, d \Ha^{n-1} - P(B_1) &=  \int_{ E} \frac{n-1}{|x|} \, dx -  \int_{ B_1} \frac{n-1}{|x|} \, dx\\
&= \int_{ E \setminus B_1} \frac{n-1}{|x|} \, dx - \int_{ B_1 \setminus E} \frac{n-1}{|x|} \, dx\,.
\end{split}
\]
This yields the equality
\begin{equation}
\label{asym.equlity}
\asym(E)^2  = D(E) -  \int_{ E \setminus B_1} \frac{n-1}{|x|} \, dx + \int_{ B_1 \setminus E} \frac{n-1}{|x|} \, dx.
\end{equation}
Let us estimate the last two terms in \eqref{asym.equlity}. Since $|E| = |B_1|$ we have
\begin{equation}
\label{definition.a}
|E \setminus B_1| =  |B_1 \setminus E| =:a.
\end{equation}
Denote by $A(R,1) = B_R \setminus B_1$ and $A(1,r) = B_1 \setminus B_r$  two annuli such that $|A(R,1)| = |A(1, r)| =a$, where $a$ is defined in \eqref{definition.a}. In other words
\[
R = \left( 1 + \frac{a}{\omega_n}\right)^{1/n} \qquad \text{and} \qquad r = \left( 1 - \frac{a}{\omega_n}\right)^{1/n}.
\]   
By construction $|A(R,1)| = |E \setminus B_1|$. Hence, we have that 
\[
\int_{ E \setminus B_1} \frac{n-1}{|x|} \, dx \leq \int_{ A(R,1)} \frac{n-1}{|x|} \, dx\,,
\]
since the weight $\frac{1}{|x|}$ gets smaller the further the set is from the unit sphere. Similarly, we have 
\[
 \int_{ B_1 \setminus E} \frac{n-1}{|x|} \, dx \geq  \int_{A(1, r)} \frac{n-1}{|x|} \, dx.
\]
Therefore we may estimate \eqref{asym.equlity} by
\begin{equation}
\label{long.calculation}
\begin{split}
\asym(E)^2  &\geq D(E) - \int_{ A(R,1)} \frac{n-1}{|x|} \, dx + \int_{ A(1, r)} \frac{n-1}{|x|} \, dx \\
&= D(E)- n\bigl[\omega_n (R^{n-1} -1) - \omega_n(1 - r^{n-1})\bigr] \\
&= D(E)  + n\omega_n \left(2 -   \left( 1 + \frac{a}{\omega_n}\right)^{\frac{n-1}{n}} -   \left( 1 - \frac{a}{\omega_n}\right)^{\frac{n-1}{n}}  \right).
\end{split}
\end{equation}

The function $f(t) = (1 + t)^{\frac{n-1}{n}}$ is uniformly concave in $[-1,1]$, i.e., 
\[
\frac{1}{2}\left(f(t) + f(s) \right) \leq f \left(\frac{t}{2} + \frac{s}{2} \right) - c_n |t-s|^2
\]
for $c_n = - \frac{1}{4} \left( \sup_{t \in (-1,1)} f''(t)\right)  = \left( \frac{1}{4n} \cdot \frac{n-1}{n} \right) 2^{\frac{-n-1}{n}} >0$. Therefore, recalling \eqref{definition.a},  we may estimate \eqref{long.calculation} by
\[
\asym(E)^2  \geq D(E) + \frac{8nc_n}{\omega_n} \, a^2 =  D(E) + {\tilde c}_n \, (|E \setminus B_1| +  |B_1 \setminus E|)^2.
\]
Since $|E \setminus B_1| +  |B_1 \setminus E| = |E \Delta B_1| $, we get 
\[
\begin{split}
(1+ {\tilde c}_n) \asym(E)^2  &\geq D(E) +{\tilde c}_n \, \left(  \int_{\partial^* E} \left( 1- \nu_{E} \cdot \frac{x}{|x|} \right) \, d \Ha^{n-1} +  |E \Delta B_1| ^2 \right) \\
&\geq D(E) + c \, \estimate(E)^2\,.
\end{split}
\]
Hence, the assertion follows. 
\end{proof}
\section*{Acknowledgment}
\noindent This research was supported by the 2008 ERC Advanced Grant 226234 ``Analytic Techniques for
Geometric and Functional Inequalities''.

\end{document}